\documentclass[11pt,reqno]{amsart}
\usepackage{fullpage}
\usepackage{mathrsfs,amssymb,graphicx,verbatim,amsmath,amsfonts}
\usepackage{paralist}
\usepackage[breaklinks,pdfstartview=FitH]{hyperref}
\usepackage{upgreek}
\renewcommand{\le}{\leqslant}
\renewcommand{\ge}{\geqslant}

\renewcommand{\setminus}{\smallsetminus}
\renewcommand{\gamma}{\upgamma}

\newcommand{\n}{\{1,\ldots,n\}}

\newcommand{\f}{\varphi}

\newcommand{\e}{\varepsilon}
\newcommand{\R}{\mathbb R}
\newcommand{\1}{\mathbf 1}

\newtheorem{theorem}{Theorem}[section]
\newtheorem{proposition}[theorem]{Proposition}
\newtheorem{lemma}[theorem]{Lemma}

\theoremstyle{remark}

\newtheorem{remark}[theorem]{Remark}
\newtheorem{question}[theorem]{Question}
\theoremstyle{definition}
\newtheorem{definition}[theorem]{Definition}

\renewcommand{\subset}{\subseteq}
\newcommand{\E}{\mathbb{ E}}

\newcommand{\F}{\mathbb F}
\newcommand{\N}{\mathbb N}

\newcommand{\eqdef}{\stackrel{\mathrm{def}}{=}}

\begin{document}

\title[Pythagorean powers of hypercubes]{Pythagorean powers of hypercubes}

\author{Assaf Naor}
\address{Mathematics Department\\ Princeton University\\ Fine Hall, Washington Road, Princeton, NJ 08544-1000, USA}
\email{naor@math.princeton.edu}
\thanks{A.~N. was supported by the NSF, the BSF, the Packard Foundation and the Simons Foundation.}

\author{Gideon Schechtman}
\address {Department of Mathematics\\
Weizmann Institute of Science\\
Rehovot 76100, Israel}
\email{gideon@weizmann.ac.il}.
\thanks{G.~S. was supported by the ISF and the BSF}
\date{}

\begin{abstract} For $n\in \N$ consider  the $n$-dimensional hypercube as equal to  the vector space $\F_2^n$, where $\F_2$ is the field of size two. Endow $\F_2^n$ with the Hamming metric, i.e., with the metric induced by the $\ell_1^n$ norm when one identifies $\F_2^n$ with $\{0,1\}^n\subset \R^n$. Denote by $\ell_2^n(\F_2^n)$ the $n$-fold Pythagorean product of $\F_2^n$, i.e., the space of all $x=(x_1,\ldots,x_n)\in \prod_{j=1}^n \F_2^n$, equipped with the metric
$$
\forall\, x,y\in \prod_{j=1}^n \F_2^n,\qquad d_{\ell_2^n(\F_2^n)}(x,y)\eqdef \sqrt{ \|x_1-y_1\|_1^2+\ldots+\|x_n-y_n\|_1^2}.
$$
It is shown here that the  bi-Lipschitz distortion  of any embedding of $\ell_2^n(\F_2^n)$ into $L_1$ is at least a constant multiple of $\sqrt{n}$. This is achieved through the following new bi-Lipschitz invariant, which is a metric version of (a slight variant of) a linear inequality of Kwapie{\'n} and Sch\"utt (1989).  Letting $\{e_{jk}\}_{j,k\in \n}$ denote the standard basis of the space of all $n$ by $n$ matrices $M_n(\F_2)$, say that a metric space $(X,d_X)$ is a  KS space if there exists $C=C(X)>0$ such that for every $n\in 2\N$, every mapping $f:M_n(\F_2)\to X$ satisfies
\begin{equation*}\label{eq:metric KS abstract}
\frac{1}{n}\sum_{j=1}^n\E\left[d_X\Big(f\Big(x+\sum_{k=1}^ne_{jk}\Big),f(x)\Big)\right]\le C \E\left[d_X\Big(f\Big(x+\sum_{j=1}^ne_{jk_j}\Big),f(x)\Big)\right],
\end{equation*}
where the expectations above are with respect to $x\in M_n(\F_2)$ and $k=(k_1,\ldots,k_n)\in \n^n$ chosen uniformly at random. It is shown here that $L_1$ is a KS space (with $C= 2e^2/(e^2-1)$, which is best possible), implying the above nonembeddability statement. Links to the Ribe program are discussed, as well as related open problems.
\end{abstract}

\maketitle

\section{Introduction}

For a metric space $(X,d_X)$ and $n\in \N$, the $n$-fold Pythagorean power of $(X,d_X)$, denoted $\ell_2^n(X)$, is the space $X^n$, equipped with metric given by setting for every $(x_1,\ldots,x_n),(y_1,\ldots,y_n)\in X$,
\begin{equation}\label{eq:def pythagorean product}
 d_{\ell_2^n(X)}\big((x_1,\ldots,x_n),(y_1,\ldots,y_n)\big)\eqdef \sqrt{d_X(x_1,y_1)^2+\ldots+d_X(x_n,y_n)^2}.
\end{equation}
For $p\in [1,\infty]$, one analogously defines the $\ell_p$ powers of $(X,d_X)$ by replacing in the right hand side of~\eqref{eq:def pythagorean product} the squares by $p$'th powers and the square root by the $p$'th root (with the obvious modification for $p=\infty$). When $(X,\|\cdot\|_X)$ is a Banach space and $p\in [1,\infty]$, one also commonly considers the Banach space $\ell_p(X)$ consisting of all the infinite sequences $x=(x_1,x_2,\ldots)\in X^{\aleph_0}$ such that $\|x\|_{\ell_p(X)}^p= \sum_{j=1}^\infty\|x_j\|_X^p<\infty$. One could give a similar definition of infinite $\ell_p$ powers for {\em pointed} metric spaces, but in the present article it will suffice to only consider $n$-fold  powers of metric spaces for finite $n\in \N$.

Throughout the ensuing discussion we shall use standard notation and terminology from Banach space theory, as in~\cite{LT77}. In particular, for $p\in [1,\infty]$ and $n\in \N$, we use the notations $\ell_p=\ell_p(\R)$ and $\ell_p^n=\ell_p^n(\R)$, and  the space $L_p$ refers to the Lebesgue function space $L_p(0,1)$. We shall also use standard notation and terminology from the theory of metric embeddings, as in~\cite{Mat02,Ost13}. In particular, a metric space $(X,d_X)$ is said to admit a bi-Lipschitz embedding into a metric space $(Y,d_Y)$ if there exists $s\in (0,\infty)$, $D\in [1,\infty)$ and a mapping $f:X\to Y$ such that
\begin{equation}\label{eq:distortion definition}
\forall\, x,y\in X,\qquad sd_X(x,y)\le d_Y(f(x),f(y))\le Dsd_X(x,y)
\end{equation}
When this happens we say that $(X,d_X)$ embeds into $(Y,d_Y)$ with distortion at most $D$.  We denote by $c_{(Y,d_Y)}(X,d_X)$ (or simply $c_Y(X), c_Y(X,d_X)$ if the metrics are clear from the context) the infimum over those $D\in [1,\infty]$ for which $(X,d_X)$ embeds into $(Y,d_Y)$ with distortion at most $D$.  When $Y=L_p$ we use the shorter notation $c_{L_p}(X,d_X)=c_p(X,d_X)$.

A folklore theorem asserts  that $\ell_2(\ell_1)$ is not isomorphic to a subspace of $L_1$. While this statement follows from a (nontrivial) gliding hump argument, we could not locate a reference to where it was first discovered; different proofs of certain stronger statements can be found in~\cite[Theorem~4.2]{Kal85}, \cite{Ray86} and~\cite[Section~3]{RS88}. More generally, $\ell_q(\ell_p)$   is not isomorphic to a subspace of $L_1$ whenever $q>p\ge 1$; the present work yields new information on this stronger statement as well, but for the sake of simplicity we shall focus for the time being only on the case of Pythagorean products.

Finite dimensional versions of the above results were discovered by Kwapie{\'n} and Sch\"utt, who proved in~\cite{KS89} that for every $n\in \N$, if $T:\ell_2^n(\ell_1^n)\to L_1$ is an injective linear mapping then necessarily $\|T\|\cdot\|T^{-1}\|\gtrsim \sqrt{n}$. Here, and in what follows, we use the convention that for $a,b\in [0,\infty)$ the notation $a\gtrsim b$ (respectively $a\lesssim b$) stands for $a\ge c b$ (respectively $a\le cb$) for some universal constant $c\in (0,\infty)$.  Below, the notation $a\asymp b$ stands for $(a\lesssim b)\wedge (b\lesssim a)$. By Cauchy--Schwarz, the identity mapping $Id:\ell_2^n(\ell_1^n)\to \ell_1^n(\ell_1^n)$ satisfies $\|Id\|\cdot\|Id^{-1}\|=\sqrt{n}$. So, the above lower bound of Kwapie{\'n} and Sch\"utt is asymptotically sharp as $n\to \infty$, up to the implicit universal constant.

By general principles, the above stated result of Kwapie{\'n} and Sch\"utt formally implies that
\begin{equation}\label{eq:L1 dist explodes}
\lim_{n\to \infty} c_1(\ell_2^n(\F_2^n))=\infty,
\end{equation}
where $\F_2^n$ is the $n$-dimensional discrete hypercube, endowed with the metric inherited from $\ell_1^n$ via the identification $\F_2^n=\{0,1\}^n\subset \R^n$.   The deduction of~\eqref{eq:L1 dist explodes} is as follows. Suppose for contradiction that $\sup_{n\in \N}c_1(\ell_2^n(\F_2^n))<\infty$. Since for every $\e>0$ every finite subset of $\ell_1$ embeds with distortion $1+\e$ into $\F_2^m$ for some $m\in \N$ (see~\cite{DL97}), it follows from our contrapositive assumption that there exists $K\in [1,\infty)$ such that for every finite subset $X\subset \ell_1$ and every $n\in \N$ we have $c_1(\ell_2^n(X))\le K$. By a standard ultrapower argument (as in~\cite{Hei80}) this implies that $\sup_{n\in \N}c_1(\ell_2^n(\ell_1^n))\le K$. Next, by using a $w^*$-G\^ateaux differentiation argument combined with the fact that $L_1^{**}$ is an $L_1(\mu)$ space (see~\cite{HM82} or~\cite[Chapter~7]{BL00})) it follows that there exists a linear operator $T:\ell_2^n(\ell_1^n)\to L_1$ with $\|T\|\cdot\|T^{-1}\|\lesssim 2K$, contradicting the lower bound of Kwapie{\'n} and Sch\"utt. This proof of~\eqref{eq:L1 dist explodes} does not yield information on the rate at which $c_1(\ell_2^n(\F_2^n))$ tends to $\infty$, a problem that we resolve here.

\begin{theorem}\label{thm:pythagorean computation intro}
We have $c_1(\ell_2^n(\F_2^n))\asymp \sqrt{n}$.
\end{theorem}
Note that if we write $Y\eqdef \ell_2^n(\F_2^n)$ then $|Y|=2^{n^2}$, and therefore by Theorem~\ref{thm:pythagorean computation intro} we have
\begin{equation}\label{eq:in terms of cardinality}
c_1(Y)\asymp \sqrt[4]{\log |Y|}.
\end{equation}
In light of~\eqref{eq:in terms of cardinality}, the following interesting open question asks whether or not $\ell_2^n(\F_2^n)$ is the finite subset of $\ell_2(\ell_1)$ that is asymptotically the furthest from a subset of $L_1$ in terms of its cardinality.
\begin{question}\label{Q:ell2ell1 aln}
{\em Suppose that $S\subset \ell_2(\ell_1)$ is finite. Is it true that $c_1(S)\lesssim \sqrt[4]{\log |S|}$?}
\end{question}

The following lemma (proved in Section~\ref{sec:embeddings} below) is a simple consequence of~\cite{ALN08}. It shows that the answer to Question~\ref{Q:ell2ell1 aln} is positive (up to lower order factors) for finite subsets $S\subset \ell_2(\ell_1)$ that are product sets, i.e., those sets of the form $S=X_1\times \ldots\times X_n\subset \ell_2^n(\ell_1)$ for some $n\in \N$ and finite subsets $X_1,\ldots,X_n\subset \ell_1$. This assertion is of course very far from a full resolution of Question~\ref{Q:ell2ell1 aln}. We conjecture that the answer to Question~\ref{Q:ell2ell1 aln} is positive, and it would be worthwhile to investigate whether or not variants of the methods used in the proof of the main result of~\cite{ALN08} are relevant here.

\begin{lemma}\label{lem:use aln} Suppose that $n\in \N$ and $X_1,\ldots,X_n\subseteq \ell_1$ are finite. Write $S=X_1\times \ldots\times X_n\subset \ell_2^n(\ell_1)$. Then
$$
c_1(S)\lesssim \sqrt[4]{\log|S|}\cdot \sqrt{\log\log |S|}=\left(\log |S|\right)^{\frac14+o(1)}.
$$
\end{lemma}

\subsection{Metric Kwapie{\'n}--Sch\"utt inequalities} In~\cite{KS89} (see also~\cite{KS85}) Kwapie{\'n} and Sch\"utt (implicitly) proved the following inequality, which holds for every $n\in \N$ and every $\{z_{jk}\}_{j,k\in \n}\subset L_1$.
\begin{equation}\label{eq:original KS}
\frac{1}{n}\sum_{j=1}^n\sum_{\e\in \{-1,1\}^n}\Big\|\sum_{k=1}^n\e_kz_{jk}\Big\|_1\lesssim \frac{1}{n!} \sum_{\pi\in S_n} \sum_{\e\in \{-1,1\}^n} \Big\|\sum_{j=1}^n\e_jz_{j\pi(j)}\Big\|_1,
\end{equation}
where $S_n$ denotes as usual the group of all permutations of $\n$.

The validity of~\eqref{eq:original KS} immediately implies the previously mentioned lower bound on the distortion of any linear embedding $T:\ell_2^n(\ell_1^n)\to L_1$.  To see this, identify from now on $\ell_2^n(\ell_1^n)$ with $M_n(\R)$ by considering for every $x=(x_1,\ldots,x_n)\in \ell_2^n(\ell_1^n)$ the matrix whose $j$'th row is $x_j\in \R^n$. With this convention,  apply~\eqref{eq:original KS} to $z_{jk}=T(e_{jk})$, where $e_{jk}$ is the $n$ by $n$ matrix whose $(j,k)$ entry equals $1$ and the rest of its entries vanish. Then for every $\e\in \{-1,1\}^n$ and $j\in \n$ we have
\begin{equation}\label{eq:lower KS}
\Big\|\sum_{k=1}^n\e_kz_{jk}\Big\|_1\ge \frac{\Big\|\sum_{k=1}^n\e_ke_{jk}\Big\|_{\ell_2^n(\ell_1^n)}}{\|T^{-1}\|} =\frac{n}{\|T^{-1}\|},
\end{equation}
and for every $\pi\in S_n$ and $\e\in \{-1,1\}^n$ we have
\begin{equation}\label{eq:upper KS}
\Big\|\sum_{j=1}^n\e_jz_{j\pi(j)}\Big\|_1\le
\|T\| \Big\|\sum_{j=1}^n\e_je_{j\pi(j)}\Big\|_{\ell_2^n(\ell_1^n)}=\|T\|\sqrt{n}.
\end{equation}
The only way for~\eqref{eq:lower KS} and~\eqref{eq:upper KS} to be compatible with~\eqref{eq:original KS} is if $\|T\|\cdot\|T^{-1}\|\gtrsim \sqrt{n}$.

In light of the above argument, it is natural to ask which Banach spaces satisfy~\eqref{eq:original KS}, i.e., to obtain an understanding of those Banach spaces $(Z,\|\cdot\|_Z)$ for which there exists $K=K(Z)\in (0,\infty)$ such that for every $n\in \N$ and every  $\{z_{jk}\}_{j,k\in \n}\subset Z$ we have
\begin{equation}\label{eq:general KS pi}
\frac{1}{n}\sum_{j=1}^n\sum_{\e\in \{-1,1\}^n}\Big\|\sum_{k=1}^n\e_kz_{jk}\Big\|_Z\le \frac{K}{n!} \sum_{\pi\in S_n} \sum_{\e\in \{-1,1\}^n} \Big\|\sum_{j=1}^n\e_jz_{j\pi(j)}\Big\|_Z.
\end{equation}
This question requires further investigation and obtaining a satisfactory characterization seems to be challenging. In particular, it seems to be unknown whether or not the Schatten trace class $S_1$ satisfies~\eqref{eq:general KS pi}. Regardless, it is clear that the requirement~\eqref{eq:general KS pi} is a local linear property, and therefore by Ribe's rigidity theorem~\cite{Rib76} it is preserved under uniform homeomorphisms of Banach spaces. In accordance with the Ribe program (see~\cite{Bal13,Nao12}) one should ask for a bi-Lipschitz invariant of metric spaces that, when restricted to the class of Banach spaces, is equivalent to~\eqref{eq:general KS pi}.

Following the methodology that was introduced by Enflo~\cite{Enf76} (see also~\cite{Gro83,BMW86}), a first attempt to obtain a bi-Lipschitz invariant that is (hopefully) equivalent to~\eqref{eq:general KS pi} is as follows. Consider those metric spaces $(X,d_X)$ for which there exists $K=K(X)\in (0,\infty)$ such that for every $n\in \N$ and every $f:M_n(\F_2)\to X$ we have
\begin{equation}\label{eq:failed attempt KS}
\frac{1}{n}\sum_{j=1}^n\sum_{x\in M_n(\F_2)}d_X\Big(f\Big(x+\sum_{k=1}^n e_{jk}\Big),f(x)\Big)\le \frac{K}{n!} \sum_{\pi\in S_n} \sum_{x\in M_n(\F_2)} d_X\Big(f\Big(x+\sum_{j=1}^n e_{j\pi(j)}\Big),f(x)\Big).
\end{equation}
If $X$ is in addition a Banach space and $\{z_{jk}\}_{j,k\in \n}\subset X$ then for $f(x)=\sum_{j=1}^n\sum_{k=1}^n (-1)^{x_{jk}}z_{jk}$ the inequality~\eqref{eq:failed attempt KS} becomes~\eqref{eq:general KS pi}. However, for every integer $n\ge 3$ no metric space that contains at least two points can satisfy~\eqref{eq:failed attempt KS} with $K< n/2$, as explained in Remark~\ref{rem:no permutation version} below. Thus, obtaining a metric characterization of the linear property~\eqref{eq:general KS pi} remains open.

We shall overcome this difficulty by first modifying the linear definition~\eqref{eq:general KS pi} so that it still implies the same nonembeddability result for $\ell_2^n(\ell_1^n)$, and at the same time we can prove that the reasoning that led to the metric inequality~\eqref{eq:failed attempt KS} now leads to an analogous inequality which does hold true for nontrivial metric spaces (specifically, we shall prove that it holds true for $L_1$).

\begin{definition}[Linear KS space]
Say that a Banach space $(Z,\|\cdot\|_Z)$ is a linear KS space if there exists $C=C(X)\in (0,\infty)$ such that for every $n\in \N$ and every  $\{z_{jk}\}_{j,k\in \n}\subset Z$ we have
\begin{equation}\label{eq:modified linear KS}
\frac{1}{n}\sum_{j=1}^n\sum_{\e\in \{-1,1\}^n}\Big\|\sum_{k=1}^n\e_kz_{jk}\Big\|_Z\le \frac{C}{n^n} \sum_{k\in \n^n} \sum_{\e\in \{-1,1\}^n} \Big\|\sum_{j=1}^n\e_jz_{jk_j}\Big\|_Z.
\end{equation}
\end{definition}
The difference between~\eqref{eq:modified linear KS} and~\eqref{eq:general KS pi} is that we replace the averaging over all permutations $\pi\in S_n$ by averaging over all mappings $\pi:\n\to \n$.  We shall see below that $L_1$ is a linear KS space.  The same reasoning that leads to~\eqref{eq:failed attempt KS} now leads us to consider the following new bi-Lipschitz invariant for metric spaces.

\begin{definition}[KS metric space]\label{def:modified metric KS}
Say that a metric space $(X,d_X)$ is a KS space if there exists $C=C(X)\in (0,\infty)$ such that for every $n\in 2\N$ and every  $f:M_n(\F_2)\to X$ we have
\begin{equation}\label{eq:modified metric KS}
 \frac{1}{n}\sum_{j=1}^n\sum_{x\in M_n(\F_2)}d_X\Big(f\Big(x+\sum_{k=1}^n e_{jk}\Big),f(x)\Big)\le \frac{C}{n^n} \sum_{k\in \n^n} \sum_{x\in M_n(\F_2)} d_X\Big(f\Big(x+\sum_{j=1}^n e_{jk_j}\Big),f(x)\Big).
\end{equation}
\end{definition}

\begin{remark}\label{rem:no odd}
The reason why in Definition~\ref{def:modified metric KS} we require~\eqref{eq:modified metric KS} to hold true only when $n\in \N$ is even is that no non-singleton metric space $(X,d_X)$ satisfies~\eqref{eq:modified metric KS} when $n\ge 3$ is an odd integer. Indeed, suppose that $a,b\in X$ are distinct and that $n\ge 3$ is an odd integer. For $x\in M_n(\F_2)$ write $\sigma(x)=\sum_{j=1}^{n-1}\sum_{k=1}^n x_{jk}\in \F_2$. Define $f: M_n(\F_2)\to X$ by setting $f(x)=a$ if $\sigma(x)=0$ and $f(x)=b$ if $\sigma(x)=1$. For every $j\in \{1,\ldots,n-1\}$ and $x\in M_n(\F_2)$ we have $\sigma(x+\sum_{k=1}^n e_{jk})=\sigma(x)+n\neq \sigma(x)$, since $n$ is odd. Consequently the left hand side of~\eqref{eq:modified metric KS} is nonzero (since $a$ and $b$ are distinct). But, for every $x\in M_n(\F_2)$ and $k\in \n^n$ we have $\sigma(x+\sum_{j=1}^n e_{jk_j})=\sigma(x)+n-1=\sigma(x)$, since $n$ is odd. Consequently the right hand side of~\eqref{eq:modified metric KS} vanishes. This parity issue is of minor significance: in Remark~\ref{rem:all n version} below we describe an inequality that is slightly more complicated than~\eqref{eq:modified metric KS} but make sense for every $n\in \N$ and in any metric space, and we show that it holds true for $L_1$-valued functions. This variant has the same nonembeddability consequences as~\eqref{eq:modified metric KS}, albeit yielding distortion lower bounds that are weaker by a constant factor.
\end{remark}

The following theorem is the main result of the present article. We shall soon see, in Section~\ref{sec:Lp nonembed} below, how it quickly implies Theorem~\ref{thm:pythagorean computation intro} (and more).

\begin{theorem}\label{thm:main L1 ineq} $L_1$ is a KS space. Namely, for all $n\in 2\N$ and every $f:M_n(\F_2)\to L_1$ we have
\begin{equation*}\label{eq:L1 has metric KS}
\sum_{j=1}^n\sum_{x\in M_n(\F_2)}\Big\|f\Big(x+\sum_{k=1}^n e_{jk}\Big)-f(x)\Big\|_1
\le \frac{2n}{n^n-(n-2)^n} \sum_{k\in \n^n} \sum_{x\in M_n(\F_2)}
\Big\|f\Big(x+\sum_{j=1}^n e_{jk_j}\Big)-f(x)\Big\|_1.
\end{equation*}
For every fixed $n\in 2\N$, the factor $2n/(n^n-(n-2)^n)$ above cannot be improved. So, $L_1$ satisfies~\eqref{eq:modified metric KS} for every $n\in 2\N$  with $$C=\sup_{n\in 2\N}\frac{2}{1-\left(1-\frac{2}{n}\right)^n}=\frac{2e^2}{e^2-1},$$
 and this value of $C$ cannot be improved.
\end{theorem}

Note that if $(Z,\|\cdot\|_Z)$ is a Banach space and $\{z_{jk}\}_{j,k\in \n}\subset Z$ then by considering the mapping $f:M_n(\F_2)\to Z$ given by $f(x)=\sum_{j=1}^n\sum_{k=1}^n (-1)^{x_{jk}}z_{jk}$ we see that if $Z$ is a KS space as a metric space then it is also a linear KS space (with the same constant $C$). We do not know whether or not the converse holds true, i.e., we ask the following interesting open question: if a Banach space $(Z,\|\cdot\|_Z)$ is a linear KS space then is it also a KS space as a metric space? Understanding which Banach spaces are linear KS spaces is a wide-open research direction. In particular, we ask whether the Schatten trace class $S_1$ is a KS space as a metric space, or even whether it is a linear KS space. There are inherent conceptual difficulties that indicate that  our proof of Theorem~\ref{thm:main L1 ineq} cannot be extended to the case of $S_1$ without a substantial new idea; see Remark~\ref{rem:Schatten difficulties} below.

Our proof of Theorem~\ref{thm:main L1 ineq} consists of simple Fourier analysis combined with a nonlinear transformation; see Section~\ref{sec:main proof} below. The simplicity of this proof indicates one of the values of generalizing linear inequalities such as~\eqref{eq:modified linear KS} to their stronger nonlinear counterparts, since this brings  genuinely nonlinear tools into play. In particular, we thus obtain a very simple proof of the linear inequality~\eqref{eq:modified linear KS} through an argument which would have probably not been found without the need to generalize~\eqref{eq:modified linear KS} to a metric inequality as part of the Ribe program.

\subsection{Embeddings of $\ell_q(\F_2^n,\|\cdot\|_p)$ into $L_1$}\label{sec:Lp nonembed} Suppose that $q>p\ge 1$. By arguing as in~\eqref{eq:lower KS} and~\eqref{eq:upper KS} one deduces from the Kwapie{\'n}--Sch\"utt inequality~\eqref{eq:original KS} that for every $n\in \N$, every injective linear mapping $T:\ell_q^n(\ell_p^n)\to L_1$ must satisfy
\begin{equation}\label{eq:Banach-Mazur}
\|T\|\cdot\|T^{-1}\|\gtrsim n^{\frac{1}{p}-\frac{1}{q}}.
\end{equation}
Note that this conclusion was stated by Kwapie{\'n} and Sch\"utt in~\cite[Crollary~3.4]{KS89} under the additional assumption that $q\le 2$, but this restriction is not necessary. By a differentiation argument (see~\cite[Chapter~7]{BL00}), it follows from~\eqref{eq:Banach-Mazur} that
\begin{equation}\label{eq:c1 lower bound full space}
c_1\left(\ell_q^n(\ell_p^n)\right)\gtrsim n^{\frac{1}{p}-\frac{1}{q}}.
\end{equation}

We previously deduced from the case $q=2$ and $p=1$ of~\eqref{eq:c1 lower bound full space} that $\lim_{n\to \infty} c_1(\ell_2^n(\F_2^n))=\infty$.
This was done in the paragraph that followed~\eqref{eq:L1 dist explodes}, relying on the fact that any finite subset of $\ell_1$ admits
 an embedding with $O(1)$ distortion into $\F_2^m$ for some $m\in \N$. The analogous assertion is not true for $p>1$, and therefore
 despite the validity of~\eqref{eq:c1 lower bound full space} it was previously unknown whether or not $\sup_{n\in \N} c_1(\ell_q(\F_2^n,\|\cdot\|_p))=\infty$.
 Our metric KS  inequality of Theorem~\ref{thm:main L1 ineq} answers this question.

\begin{theorem}\label{thm:sharp pq} Suppose that $1\le p<q$ and $n\in \N$ then
\begin{equation}\label{eq:discrete distortion sharp}
c_1\left(\ell_q^n\left(\F_2^n,\|\cdot\|_p\right)\right)\asymp n^{\frac{1}{p}-\frac{1}{q}}.
\end{equation}
\end{theorem}

It is worthwhile to note here that while Theorem~\ref{thm:sharp pq} yields a sharp asymptotic evaluation of $c_1\left(\ell_q^n\left(\F_2^n,\|\cdot\|_p\right)\right)$, the corresponding bound~\eqref{eq:c1 lower bound full space} in the continuous setting is not always sharp. Specifically, in Section~\ref{sec:embeddings} we explain that
\begin{equation}\label{eq:cont dist}
c_1\left(\ell_q^n(\ell_p^n)\right) \asymp \left\{\begin{array}{ll}
n^{\frac{1}{p}-\frac{1}{q}} &\mathrm{if\ } 1\le p<q\ \mathrm{and}\ p\le 2,\\
n^{1-\frac{1}{p}-\frac{1}{q}} &\mathrm{if\ } 2\le p<q.
\end{array}\right.
\end{equation}
From~\eqref{eq:discrete distortion sharp} and~\eqref{eq:cont dist} we see that  if $1\le p<q$ and $p\le 2$ then $c_1\left(\ell_q^n\left(\F_2^n,\|\cdot\|_p\right)\right)\asymp c_1\left(\ell_q^n(\ell_p^n)\right)$, while if  $2<p<q$ then since $1/p-1/q<1-1/p-1/q$ we have  $c_1\left(\ell_q^n\left(\F_2^n,\|\cdot\|_p\right)\right)=o\left(c_1\left(\ell_q^n(\ell_p^n)\right)\right)$.

The upper bound on $c_1(\ell_q^n\left(\F_2^n,\|\cdot\|_p\right))$ that appears in~\eqref{eq:discrete distortion sharp} will be proven in Section~\ref{sec:embeddings}.
We shall now show how the lower bound on
$c_1(\ell_q^n\left(\F_2^n,\|\cdot\|_p\right))$ that appears in~\eqref{eq:discrete distortion sharp} quickly
follows from Theorem~\ref{thm:main L1 ineq}. This will also establish Theorem~\ref{thm:pythagorean computation intro}  as a special case. So, suppose that $D\in [1,\infty)$ and $f:M_n(\F_2)\to L_1$ satisfies $\|x-y\|_{\ell_q^n(\ell_p^n)}\le \|f(x)-f(y)\|_1\le D\|x-y\|_{\ell_q^n(\ell_p^n)}$ for every $x,y\in M_n(\F_2)$. Our goal is to bound $D$ from below. Since $\ell_q^n\left(\F_2^n,\|\cdot\|_p\right)$ contains an isometric copy of $\ell_q^{n-1}\left(\F_2^{n-1},\|\cdot\|_p\right)$, we may assume that $n$ is even. By Theorem~\ref{thm:main L1 ineq} applied to $f$ we have
\begin{equation}\label{eq:quote theorem}
\frac{1}{n}\sum_{j=1}^n\sum_{x\in M_n(\F_2)}\Big\|f\Big(x+\sum_{k=1}^n e_{jk}\Big)-f(x)\Big\|_1
\lesssim \frac{1}{n^n}\sum_{k\in \n^n} \sum_{x\in M_n(\F_2)}
\Big\|f\Big(x+\sum_{j=1}^n e_{jk_j}\Big)-f(x)\Big\|_1.
\end{equation}
But,
\begin{equation}\label{eq:contrast1}
\sum_{j=1}^n\sum_{x\in M_n(\F_2)}\Big\|f\Big(x+\sum_{k=1}^n e_{jk}\Big)-f(x)\Big\|_1\ge 2^{n^2}\sum_{j=1}^n\Big\|\sum_{k=1}^n e_{jk}\Big\|_{\ell_q^n(\ell_p^n)}=2^{n^2} n^{1+\frac{1}{p}},
\end{equation}
and
\begin{equation}\label{eq:contrast2}
\sum_{k\in \n^n} \sum_{x\in M_n(\F_2)}
\Big\|f\Big(x+\sum_{j=1}^n e_{jk_j}\Big)-f(x)\Big\|_1\le D2^{n^2}\sum_{k\in \n^n}\Big\|\sum_{j=1}^n e_{jk_j}\Big\|_{\ell_q^n(\ell_p^n)}=D2^{n^2}n^{n+\frac{1}{q}}.
\end{equation}
For~\eqref{eq:quote theorem} to be compatible with~\eqref{eq:contrast1} and~\eqref{eq:contrast2} we must have $D\gtrsim n^{\frac{1}{p}-\frac{1}{q}}$. \qed

\begin{remark}\label{rem:mn intro}
We only discussed $L_1$ embeddings of $\ell_q^n\left(\F_2^n,\|\cdot\|_p\right)$, but it is natural to also ask about embeddings of  $\ell_q^m\left(\F_2^n,\|\cdot\|_p\right)$. However, it turns out that the case $m=n$ is the heart of the matter,  i.e., the $L_1$ distortion of  $\ell_q^m\left(\F_2^n,\|\cdot\|_p\right)$ is up to a constant factor the same as the $L_1$ distortion of   $\ell_q^k\left(\F_2^k,\|\cdot\|_p\right)$ with $k=\min\{m,n\}$; see Remark~\ref{rem:mn proof} below.
\end{remark}

\section{Proof of Theorem~\ref{thm:main L1 ineq}}\label{sec:main proof}

The stated sharpness of Theorem~\ref{thm:main L1 ineq} is simple: consider the function $\f:M_n(\F_2)\to \R$ given by  $\f(x)= (-1)^{x_{11}+\ldots+x_{nn}}$. For this choice of $\f$ we have $\f(x+\sum_{k=1}^n e_{jk})=-\f(x)\in \{-1,1\}$ for every $x\in M_n(\F_2)$ and $j\in \n$. Consequently,
\begin{equation}\label{eq:sharpness1}
\sum_{j=1}^n\sum_{x\in M_n(\F_2)}\Big|\f\Big(x+\sum_{k=1}^n e_{jk}\Big)-\f(x)\Big|=n2^{n^2+1}.
\end{equation}
Also, for every $x\in M_n(\F_2)$ and $k\in \n^n$ we have $\f(x+\sum_{j=1}^n e_{jk_j})=(-1)^{\ell(k)} \f(x)$, where $\ell(k)=|\{j\in \n:\ k_j=j\}|=\sum_{j=1}^n \1_{\{k_j=j\}}$. Consequently,
\begin{multline}\label{eq:sharpness 2}
\sum_{k\in \n^n}\sum_{x\in M_n(\F_2)}\Big|\f\Big(x+\sum_{j=1}^n e_{jk_j}\Big)-\f(x)\Big|
=2^{n^2}\sum_{k\in \n^n}\left(1-(-1)^{\sum_{j=1}^n \1_{\{k_j=j\}}}\right)\\=2^{n^2}n^n-2^{n^2}\prod_{j=1}^n \sum_{k=1}^n (-1)^{\1_{\{k=j\}}}=2^{n^2}\left(n^n-(n-2)^n\right).
\end{multline}
The identities~\eqref{eq:sharpness1} and~\eqref{eq:sharpness 2} demonstrate that for every fixed $n\in \N$ the factor $2n/(n^n-(n-2)^n)$ in Theorem~\ref{thm:main L1 ineq} cannot be replaced by any strictly smaller number.

Passing now to the proof of Theorem~\ref{thm:main L1 ineq}, we will actually prove the following statement, the case $p=1$ of which is Theorem~\ref{thm:main L1 ineq} itself.

\begin{theorem}\label{thm:p version for snowflake}
Suppose that $p\in (0,2]$ and $n\in 2\N$. Then for every $f:M_n(\F_2)\to L_p$ we have
\begin{equation*}
\sum_{j=1}^n\sum_{x\in M_n(\F_2)}\Big\|f\Big(x+\sum_{k=1}^n e_{jk}\Big)-f(x)\Big\|_p^p
\le \frac{2n}{n^n-(n-2)^n} \sum_{k\in \n^n} \sum_{x\in M_n(\F_2)}
\Big\|f\Big(x+\sum_{j=1}^n e_{jk_j}\Big)-f(x)\Big\|_p^p.
\end{equation*}
\end{theorem}

\begin{proof}
By a classical theorem of Schoenberg~\cite{Sch38}, the metric space $(L_p,\|x-y\|_p^{p/2})$ admits an isometric embedding into $L_2$. Since the desired inequality is purely metric, i.e., it involves only distances between various values of $f$, it suffices to prove it for $p=2$ and then apply it to the composition of $f$ with the Schoenberg isometry so as to deduce the desired inequality for general $p\in (0,2]$. In order to prove the case $p=2$, it suffices to prove the desired inequality when $f$ is real-valued (deducing the case of $L_2$-valued $f$ by integrating the resulting point-wise inequality).

Suppose then that $f:M_n(\F_2)\to \R$. We shall use below standard Fourier-analytic arguments on $M_n(\F_2)$, considered as a vector space (of dimension $n^2$) over $\F_2$. Specifically, one can write
$$
f(x)=\sum_{A_1,\ldots,A_n\subset \n} \widehat{f}(A_1,\ldots,A_n)(-1)^{\sum_{j=1}^n\sum_{k\in A_j}x_{jk}},
$$
where for every $A_1,\ldots,A_n\subset \n$,
$$
\widehat{f}(A_1,\ldots,A_n)\eqdef \frac{1}{2^{n^2}}\sum_{x\in M_n(\F_2)} (-1)^{\sum_{j=1}^n\sum_{k\in A_j}x_{jk}}f(x).
$$
Then, for every $x\in M_n(\F_2)$ and $j\in \n$ we have
\begin{align*}
f\Big(x+\sum_{k=1}^n e_{jk}\Big)-f(x)&=\sum_{A_1,\ldots,A_n\subset \n} \widehat{f}(A_1,\ldots,A_n)\left((-1)^{|A_j|}-1\right)(-1)^{\sum_{s=1}^n\sum_{k\in A_s}x_{sk}}\\
&=-2\sum_{\substack{A_1,\ldots,A_n\subset \n\\ |A_j|\equiv 1\mod 2}} \widehat{f}(A_1,\ldots,A_n)(-1)^{\sum_{s=1}^n\sum_{k\in A_s}x_{sk}}.
\end{align*}
Hence, by the orthogonality of the functions $\{x\mapsto (-1)^{\sum_{s=1}^n\sum_{k\in A_s}x_{sk}}\}_{A_1,\ldots,A_n\subset \n}$ on $M_n(\F_2)$,
\begin{align}\label{eq:lower identity}
\sum_{j=1}^n\sum_{x\in M_n(\F_2)}&\Big(f\Big(x+\sum_{k=1}^n e_{jk}\Big)-f(x)\Big)^2\nonumber \\  &=2^{n^2+2}\sum_{A_1,\ldots,A_n\subset \n}\left|\left\{j\in \n:\ |A_j|\equiv 1\mod 2\right\}\right|\widehat{f}(A_1,\ldots,A_n)^2.
\end{align}
At the same time, for every $x\in M_n(\F_2)$ and $k\in \n^n$ we have
$$
f\Big(x+\sum_{j=1}^n e_{jk_j}\Big)-f(x)=\sum_{A_1,\ldots,A_n\subset \n} \widehat{f}(A_1,\ldots,A_n)\left((-1)^{\sum_{j=1}^n\1_{A_j}(k_j)}-1\right)(-1)^{\sum_{j=1}^n\sum_{k\in A_j}x_{jk}}.
$$
Using orthogonality again, we therefore have
\begin{align}\label{eq:upper identity}
\sum_{k\in \n^n} &\sum_{x\in M_n(\F_2)}
\Big(f\Big(x+\sum_{j=1}^n e_{jk_j}\Big)-f(x)\Big)^2\nonumber \\ \nonumber &=2^{n^2}\sum_{k\in \n^n}\sum_{A_1,\ldots,A_n\subset \n} \widehat{f}(A_1,\ldots,A_n)^2\left((-1)^{\sum_{j=1}^n\1_{A_j}(k_j)}-1\right)^2\\
 &= 2^{n^2+1} \sum_{A_1,\ldots,A_n\subset \n} \widehat{f}(A_1,\ldots,A_n)^2 \sum_{k\in \n^n}\left(1-(-1)^{\sum_{j=1}^n\1_{A_j}(k_j)}\right).
\end{align}

Fixing $A_1,\ldots,A_n\subset \n$, denote $S\eqdef \left\{j\in \n:\ |A_j|\equiv 1\mod 2\right\}$. Since $n$ is even, if $j\in S$ then $|A_j|\in \{1,\ldots,n-1\}$, and consequently $|2|A_j|-n|\le n-2$. Hence,
\begin{multline}\label{eq:use n even}
\sum_{k\in \n^n}\left(1-(-1)^{\sum_{j=1}^n\1_{A_j}(k_j)}\right)=n^n-\prod_{j=1}^n \sum_{k=1}^n (-1)^{\1_{A_j}(k)}\\ =n^n-\prod_{j=1}^n\left(n-2|A_j|\right)\ge n^n-\prod_{j=1}^n\big|2|A_j|-n\big|\ge n^n-n^{n-|S|}(n-2)^{|S|}.
\end{multline}
Since the mapping $|S|\mapsto \left(n^n-n^{n-|S|}(n-2)^{|S|}\right)/|S|$ is decreasing in $|S|$, it follows from~\eqref{eq:use n even} that
\begin{equation}\label{eq:to contrast n even}
\sum_{k\in \n^n}\left(1-(-1)^{\sum_{j=1}^n\1_{A_j}(k_j)}\right)\ge \frac{n^n-(n-2)^n}{n}\left|\left\{j\in \n:\ |A_j|\equiv 1\mod 2\right\}\right|.
\end{equation}
The desired inequality now follows by substituting~\eqref{eq:to contrast n even} into~\eqref{eq:upper identity} and recalling~\eqref{eq:lower identity}.
\end{proof}

\begin{remark}\label{rem:cut cone}
By the cut-cone decomposition of $L_1$ metrics (see e.g.~\cite{DL97}), the inequality of Theorem~\eqref{thm:main L1 ineq} is equivalent to the following (also sharp) isoperimetric-type inequality. For every and $n\in 2\N$ and $S\subseteq M_n(\F_2)$ we have
\begin{equation}\label{eq:ispoerimetric}
\sum_{j=1}^n \Big|\Big\{x\in S:\ x+\sum_{k=1}^n e_{jk}\notin S\Big\}\Big|\le \frac{2n}{n^n-(n-2)^n}\sum_{k\in \n^n} \Big|\Big\{x\in S:\ x+\sum_{j=1}^n e_{jk_j}\notin S\Big\}\Big|.
\end{equation}
Due to the simplicity of our proof of Theorem~\ref{thm:main L1 ineq}, we did not attempt to obtain a direct combinatorial proof of~\eqref{eq:ispoerimetric}, though we believe that this should be doable (and potentially instructive). We also did not attempt to characterize the equality cases in~\eqref{eq:ispoerimetric}.
\end{remark}

\begin{remark}\label{rem:all n version} In Remark~\ref{rem:no odd} we have seen that~\eqref{eq:modified metric KS} can hold true only if $n\in \N$ is even. However, this parity issue can remedied through the following (sharp) inequality, which holds true for every $n\in \N$, every $p\in (0,2]$ and every $f:M_n(\F_2)\to L_p$.
\begin{align}\label{eq:modified with y}
\nonumber \frac{1}{n}\sum_{j=1}^n&\sum_{x\in M_n(\F_2)}\sum_{y\in \F_2^n}\Big\|f\Big(x+\sum_{k=1}^n y_je_{jk}\Big)-f(x)\Big\|_p^p
\\\nonumber & \le \frac{2}{1-\left(1-\frac{1}{n}\right)^n}\cdot \frac{1}{n^n2^{n^2+n}} \sum_{k\in \n^n} \sum_{x\in M_n(\F_2)}\sum_{y\in \F_2^n}
\Big\|f\Big(x+\sum_{j=1}^n y_je_{jk_j}\Big)-f(x)\Big\|_p^p\\
&\le \frac{2e}{e-1}\cdot \frac{1}{n^n2^{n^2+n}} \sum_{k\in \n^n} \sum_{x\in M_n(\F_2)}\sum_{y\in \F_2^n}
\Big\|f\Big(x+\sum_{j=1}^n y_je_{jk_j}\Big)-f(x)\Big\|_p^p.
\end{align}
The distortion lower bounds that we obtained as a consequence of Theorem~\ref{thm:main L1 ineq} also follow mutatis mutandis from~\eqref{eq:modified with y}, though they are weaker by a constant factor.

To prove~\eqref{eq:modified with y}, note that, exactly as in the beginning of the proof of Theorem~\ref{thm:p version for snowflake}, it suffices to prove~\eqref{eq:modified with y} when $p=2$ and $f:M_n(\F_2)\to \R$. Now, argue as in~\eqref{eq:upper identity} to obtain the following identity.
\begin{align}\label{eq:new versuon upper}
\sum_{k\in \n^n} &\sum_{x\in M_n(\F_2)}\sum_{y\in \F_2^n}
\Big(f\Big(x+\sum_{j=1}^n y_je_{jk_j}\Big)-f(x)\Big)^2\nonumber \\
 &= 2^{n^2+1} \sum_{A_1,\ldots,A_n\subset \n} \widehat{f}(A_1,\ldots,A_n)^2 \sum_{k\in \n^n}\sum_{y\in \F_2^n}\left(1-(-1)^{\sum_{j=1}^ny_j\1_{A_j}(k_j)}\right).
\end{align}
For every $A_1,\ldots,A_n\subset \n$ and $k\in \n^n$ we have
$$
\sum_{y\in \F_2^n}\left(1-(-1)^{\sum_{j=1}^ny_j\1_{A_j}(k_j)}\right)=\left\{
\begin{array}{ll}
2^n &\mathrm{if\ }k_j\in A_j\ \mathrm{for\ some\ }j\in \n,\\ 0 &\mathrm{otherwise}.\end{array}\right.
$$
Hence, denoting $T\eqdef \{j\in \n:\ A_j\neq \emptyset\}\supseteq \{j\in \n:\ |A_j|\equiv 1\mod 2\}\eqdef S$,
\begin{multline}\label{eq:T}
\sum_{k\in \n^n}\sum_{y\in \F_2^n}\left(1-(-1)^{\sum_{j=1}^ny_j\1_{A_j}(k_j)}\right)=2^n\sum_{k\in \n^n}\left(1-\1_{\{\forall\, j\in\n,\quad k_j\notin A_j\}}\right)\\= 2^n\Big(n^n-\prod_{j=1}^n(n-|A_j|)^n\Big)\ge 2^n \Big(n^n-n^{n-|T|}(n-1)^{|T|}\Big)\ge 2^n\left(n^n-(n-1)^n\right)\frac{|T|}{n}.
\end{multline}
Consequently, by~\eqref{eq:new versuon upper} and~\eqref{eq:T}, combined with the fact that $|T|\ge |S|$, we have
\begin{align*}
&\sum_{k\in \n^n} \sum_{x\in M_n(\F_2)}\sum_{y\in \F_2^n}
\Big(f\Big(x+\sum_{j=1}^n y_je_{jk_j}\Big)-f(x)\Big)^2\\&\ge \frac{2^{n^2+n+1}\left(n^n-(n-1)^n\right)}{n}\sum_{A_1,\ldots,A_n\subset \n} \left|\left\{j\in \n:\ |A_j|\equiv 1\mod 2\right\}\right|\widehat{f}(A_1,\ldots,A_n)^2\\
&\stackrel{\eqref{eq:lower identity}}{=} \frac{2^{n}\left(n^n-(n-1)^n\right)}{2n}\sum_{j=1}^n\sum_{x\in M_n(\F_2)}\Big(f\Big(x+\sum_{k=1}^n e_{jk}\Big)-f(x)\Big)^2.
\end{align*}
This completes the proof of~\eqref{eq:modified with y}.
\end{remark}

\begin{remark}\label{rem:no permutation version}
As stated in the Introduction, the ``vanilla" metric Kwapie{\'n}--Sch\"utt inequality~\eqref{eq:failed attempt KS} cannot hold true in any non-singleton metric space $(X,d_X)$. To see this, note first that we have already seen in Remark~\ref{rem:no odd} that if $n\in \N$ is odd then~\eqref{eq:failed attempt KS} fails to hold true for any $K>0$. So, suppose that $n\ge 4$ is an even integer. It suffice to deal with $X=\{-1,1\}\subset \R$.  Define $\psi:M_n(\F_2)\to \{-1,1\}$ by
$
\psi(x)= (-1)^{x_{11}+\sum_{j=2}^{n}\sum_{k=4}^{n} x_{jk}}.
$
For every $x\in M_n(\F_2)$ we have $\psi(x+\sum_{k=1}^n e_{1k})=-\psi(x)$, and for $j\in \{2,\ldots,n\}$ we have $\psi(x+\sum_{k=1}^n e_{jk})=(-1)^{n-3}\psi(x)=-\psi(x)$, since $n$ is even. Consequently,
\begin{equation}\label{eq:psi left}
\sum_{j=1}^n\sum_{x\in M_n(\F_2)}\Big|\psi\Big(x+\sum_{k=1}^n e_{jk}\Big)-\psi(x)\Big|=n2^{n^2+1}.
\end{equation}
At the same time, for every $x\in M_n(\F_2)$ and $\pi\in S_n$ we have
\begin{equation}\label{eq:psi shift identity}
\psi \Big(x+\sum_{j=1}^n e_{j\pi(j)}\Big)=(-1)^{\1_{\{\pi(1)=1\}}+\sum_{j=2}^n\1_{\{\pi(j)\ge 4\}}}\psi(x).
\end{equation}
If $\pi(1)=1$ then $\{2,3\}\subset \{\pi(2),\ldots,\pi(n)\}$ and therefore $\1_{\{\pi(1)=1\}}+\sum_{j=2}^n\1_{\{\pi(j)\ge 4\}}=n-2$ is even. If $\pi(1)\in \n\setminus\{1,2,3\}$ then $\{\pi(2),\ldots,\pi(n)\}\supseteq \{1,2,3\}$ and consequently we have that $\1_{\{\pi(1)=1\}}+\sum_{j=2}^n\1_{\{\pi(j)\ge 4\}}=n-4$ is even. In the remaining case $\pi(1)\in \{2,3\}$ we have that $\1_{\{\pi(1)=1\}}+\sum_{j=2}^n\1_{\{\pi(j)\ge 4\}}=n-3$ is odd. Hence, by~\eqref{eq:psi shift identity} we have
\begin{equation}\label{eq:n-1 factorial}
\sum_{\pi\in S_n} \sum_{x\in M_n(\F_2)} \Big|\psi\Big(x+\sum_{j=1}^n e_{j\pi(j)}\Big)-\psi(x)\Big|=2^{n^2+1}\left|\big\{\pi\in S_n:\ \pi(1)\in \{2,3\}\big\}\right|=2^{n^2+2}(n-1)!.
\end{equation}
By contrasting~\eqref{eq:psi left} with~\eqref{eq:n-1 factorial} we see that if~\eqref{eq:failed attempt KS} holds true then necessarily $K\ge n/2$.
\end{remark}

\section{Uniform and coarse nonembeddability}\label{sec:uniform} A metric space $(X,d_X)$ is said to admit a uniform embedding into a Banach space $(Z,\|\cdot\|_Z)$ if there exists an injective mapping $f:X\to Z$ and nondecreasing functions $\alpha,\beta:(0,\infty)\to (0,\infty)$ with $\lim_{t\to 0}\beta(t)=0$ such that $\alpha(d_X(a,b))\le \|f(a)-f(b)\|_Z\le\beta (d_X(a,b))$ for all distinct $a,b\in X$. Similarly, $(X,d_X)$ is said to admit a coarse embedding into a Banach space $(Z,\|\cdot\|_Z)$ if there exists an injective mapping $f:X\to Z$ and nondecreasing functions $\alpha,\beta:(0,\infty)\to (0,\infty)$ with $\lim_{t\to\infty}\alpha(t)=\infty$ for which $\alpha(d_X(a,b))\le \|f(a)-f(b)\|_Z\le\beta (d_X(a,b))$ for all distinct $a,b\in X$.

The space $\ell_2(\ell_1)$ does not admit a uniform or coarse embedding into $L_1$. Indeed, by~\cite{AMM85} in the case of uniform embeddings and by~\cite{Ran06} in the case of coarse embeddings, this would imply that  $\ell_2(\ell_1)$ is linearly isomorphic to a subspace of $L_0$, which is proved to be impossible in~\cite[Theorem~4.2]{Kal85}.

Theorem~\ref{thm:main L1 ineq} yields a new proof that $\ell_2(\ell_1)$ does not admit a uniform or coarse embedding into $L_1$. Indeed, suppose that $\alpha,\beta:(0,\infty)\to (0,\infty)$ are nondecreasing and $f:\ell_2(\ell_1)\to L_1$ satisfies
\begin{equation}\label{eq:alpha beta assumption}
\forall\, x,y\in \ell_2(\ell_1),\qquad \alpha\left(\|f(x)-f(y)\|_{\ell_2(\ell_1)}\right)\le \|f(x)-f(y)\|_1\le \beta\left(\|f(x)-f(y)\|_{\ell_2(\ell_1)}\right).
\end{equation}
For every $s\in (0,\infty)$ and $n\in 2\N$, apply Theorem~\ref{thm:main L1 ineq} to the mapping $f_s:M_n(\F_2)\to L_1$ given by $f_s(x)=f(sx)$. The resulting inequality, when combined with~\eqref{eq:alpha beta assumption}, implies that $\alpha(sn)\lesssim \beta(s\sqrt{n})$. Choosing $s=1/\sqrt{n}$ shows that $\alpha(\sqrt{n})\lesssim\beta (1)$, so $f$ is not a coarse embedding, and choosing $s=1/n$ shows that $\beta(1/\sqrt{n})\gtrsim \alpha(1)>0$, so $f$ is not a uniform embedding.

Observe that since, by~\cite{Kad58}, $L_p$ is isometric to a subset of $L_1$ when $p\in [1,2]$, the above discussion implies that  $\ell_2(\ell_1)$ does not admit a uniform or coarse embedding into $L_p$ for every $p\in [1,2]$. Passing now to an examination of the uniform and coarse embeddability of $\ell_2(\ell_1)$ into $L_p$ for $p>2$, observe first that since, by~\cite{AMM85} and~\cite{JR06}, when $p>2$ there is no uniform or coarse embedding of $L_p$ into $L_1$, the fact that  $\ell_2(\ell_1)$ does not admit a uniform or coarse embedding into $L_1$ does not imply that $\ell_2(\ell_1)$ fails to admit such an embedding into $L_p$. An inspection of the above argument reveals that in order to show that $\ell_2(\ell_1)$ does not admit a uniform or coarse embedding into $L_p$ it would suffice to establish the following variant of Theorem~\ref{thm:p version for snowflake} when $p>2$: there exits $C_p,\theta_p\in (0,\infty)$ such that  for every $n\in 2\N$, every $f:M_n(\F_2)\to L_p$ satisfies
\begin{equation}\label{eq:no p>2}
\frac{1}{n}\sum_{j=1}^n\sum_{x\in M_n(\F_2)}\Big\|f\Big(x+\sum_{k=1}^n e_{jk}\Big)-f(x)\Big\|_p^{\theta_p}
\le \frac{C_p}{n^n}\sum_{k\in \n^n} \sum_{x\in M_n(\F_2)}
\Big\|f\Big(x+\sum_{j=1}^n e_{jk_j}\Big)-f(x)\Big\|_p^{\theta_p}.
\end{equation}
However, no such extension of Theorem~\ref{thm:p version for snowflake} to the range $p>2$ is possible. Indeed, since $\ell_2$ is linearly isometric to a subspace of $L_p$, we may fix a linear isometry $U:\ell_p^n(\ell_2^n)\to L_p$. Define $f:M_n(\F_2)\to L_p$ by $f(x)=\sum_{j=1}^n\sum_{k=1}^n x_{jk} U(e_{jk})$. For this choice of $f$, \eqref{eq:no p>2} becomes
\begin{equation*}
2^{n^2}n^{\frac{\theta_p}{2}}=\frac{1}{n}\sum_{j=1}^n\sum_{x\in M_n(\F_2)}\Big\|\sum_{k=1}^n e_{jk}\Big\|_{\ell_p^n(\ell_2^n)}^{\theta_p} \stackrel{\eqref{eq:no p>2}}{\le}
\frac{C_p}{n^n}\sum_{k\in \n^n} \sum_{x\in M_n(\F_2)}
\Big\|\sum_{j=1}^n e_{jk_j}\Big\|_{\ell_p^n(\ell_2^n)}^{\theta_p}
= 2^{n^2}C_p n^{\frac{\theta_p}{p}},
\end{equation*}
which is a contradiction for large enough $n\in 2\N$, because $p>2$. Thus, it was crucial to assume in Theorem~\ref{thm:p version for snowflake}  that $p\le 2$. When $p\ge 4$, this is  accentuated by the following proposition.

\begin{proposition}\label{prop:p>4}
For every $p\ge 4$ the exists a mapping $F_p:\ell_2(\ell_1)\to L_p$ that satisfies
\begin{equation}\label{eq:Fp requirement}
\forall\, x,y\in \ell_2(\ell_1),\qquad \|F_p(x)-F_p(y)\|_p=\|x-y\|_{\ell_2(\ell_1)}^{\frac{2}{p}}.
\end{equation}
Thus $\ell_2(\ell_1)$ admits an embedding into $L_p$ that is both uniform and coarse.
\end{proposition}

\begin{proof} Fix $T:\ell_1\to L_2$ such that
\begin{equation}\label{eq:T negative type}
\forall\, x,y\in \ell_1,\qquad \|T(x)-T(y)\|_2=\sqrt{\|x-y\|_1}.
\end{equation}
See~\cite{DL97}  for the existence of such $T$ (an explicit formula for $T$ appear in~\cite[Section~3]{Nao10}). By a classical theorem of Schoenberg~\cite{Sch38}, since $4/p\le 1$ there exists a mapping $\sigma_p:L_2\to L_2$ that satisfies
\begin{equation}\label{eq:4/p}
\forall\, x,y\in L_2\qquad \|\sigma_p(x)-\sigma_p(y)\|_2=\|x-y\|_2^{\frac{4}{p}}.
\end{equation}
Fix also an isometric embedding $S:L_2\to L_p$ and define $F_p:\ell_2(\ell_1)\to \ell_p(L_p)\cong L_p$ by
$$
\forall\, x\in \ell_2(\ell_1),\qquad F_p(x)\eqdef \left(S\circ \sigma_p\circ T(x_j)\right)_{j=1}^\infty.
$$
Then, for every $x,y\in \ell_2(\ell_1)$ we have
\begin{multline*}
\|F_p(x)-F_p(y)\|_{\ell_p(L_p)} =\Big(\sum_{j=1}^\infty \left\|S(\sigma_p(T(x_j)))-S(\sigma_p(T(y_j)))\right\|_p^p\Big)^{\frac{1}{p}}\\=
\Big(\sum_{j=1}^\infty \left\|\sigma_p(T(x_j))-\sigma_p(T(y_j))\right\|_2^p\Big)^{\frac{1}{p}}\stackrel{\eqref{eq:4/p}}{=}\Big(\sum_{j=1}^\infty \left\|T(x_j)-T(y_j)\right\|_2^4\Big)^{\frac{1}{p}}\stackrel{\eqref{eq:T negative type}}{=}\Big(\sum_{j=1}^\infty \left\|x_j-y_j\right\|_1^2\Big)^{\frac{1}{p}},
\end{multline*}
which is precisely the desired requirement~\eqref{eq:Fp requirement}.
\end{proof}

The above proof of Proposition~\ref{prop:p>4} used the fact that $p\ge 4$ in order for~\eqref{eq:4/p} to hold true. It is therefore natural to ask Question~\ref{Q:2<p<4} below. Analogous questions could be asked for uniform and coarse embeddings of $\ell_{p_1}(\ell_{p_2})$ into $L_{p_3}$ (or even into $\ell_{p_3}(L_{p_4})$), and various partial results could be obtained using similar arguments (at times with the embedding in~\eqref{eq:4/p} replaced by the embedding of~\cite[Remark~5.10]{MN04}). We shall not pursue this direction here because it yields incomplete results.

\begin{question}\label{Q:2<p<4}
Suppose that $2<p<4$. Does $\ell_2(\ell_1)$ admit a uniform or coarse embedding into $L_p$?
\end{question}

\begin{remark}\label{rem:Schatten difficulties}
In the Introduction we asked whether or not the Schatten trace class $S_1$ is a KS metric space. The approach of Section~\ref{sec:main proof} seems inherently insufficient to address this question. Indeed, we treated $L_1$ by relating its metric to Hilbert space through the isometric embedding of $(L_1,\sqrt{\|x-y\|_1})$, while $S_1$ is not even uniformly homeomorphic to a subset of Hilbert space (this follows from~\cite{AMM85} combined with e.g.~\cite{Kal85} and the classical linear nonembeddability result of~\cite{McC67}). For this reason we believe that asking about the validity of~\eqref{eq:modified metric KS} in $S_1$ is worthwhile beyond its intrinsic interest, as a potential step towards addressing more general situations in which one cannot reduce the question to (nonlinear) Hilbertian considerations.
\end{remark}

\section{Embeddings}\label{sec:embeddings}

In this section we shall justify the remaining (simple) embedding statements that were given without proof in the Introduction, starting with the proof of Lemma~\ref{lem:use aln}.

\begin{proof}[Proof of Lemma~\ref{lem:use aln}] Recall that we are given $n\in \N$, finite subsets $X_1,\ldots,X_n\subset \ell_1$, and we denote
$S=X_1\times\ldots\times X_n\subset \ell_2^n(\ell_1)$. Thus $|S|=\prod_{j=1}^n|X_i|$. We may assume without loss of generality that $|X_j|>1$ for all $j\in \n$.
Write
\begin{equation}\label{eq:def J}
J\eqdef \left\{j\in \n:\ |X_j|> \exp\left(\frac{\sqrt{\log |S|}}{\log\log |S|}\right)\right\}.
\end{equation}
Then
\begin{equation}\label{eq:J upper}
|S|\ge \prod_{j\in J}|X_j| > \exp\left(\frac{|J|\sqrt{\log |S|}}{\log\log |S|}\right)\implies |J|<\sqrt{\log |S|}\log\log |S|.
\end{equation}

By the main result of~\cite{ALN08}, for every $j\in \n$ there exists $f_j:X_j\to \ell_2$ such that
\begin{equation}\label{eq:fj assumption}
\forall\, u,v\in X_j,\qquad \|u-v\|_1\le \|f_j(u)-f_j(v)\|_2\lesssim \sqrt{\log |X_j|}\log\log |X_j|\cdot \|u-v\|_1.
\end{equation}
We shall fix from now on an isometric embedding $T:\ell_2^{\n\setminus J}(\ell_2)\to L_1$.

Define $\phi:S\to (\ell_1^J(\ell_1)\oplus L_1)_1$, where $(\ell_1^J(\ell_1)\oplus L_1)_1$ is the corresponding $\ell_1$-direct sum, by setting
$$
\phi(u)\eqdef \left((u_j)_{j\in J}\right)\oplus T\left((f_j(u_j))_{j\in \n\setminus J}\right).
$$
Then for every $u,v\in S$ we have
\begin{multline}\label{eq:phi in image}
\|\phi(u)-\phi(v)\|_{(\ell_1^J(\ell_1)\oplus L_1)_1}=\sum_{j\in J} \|u_j-v_j\|_1+\bigg(\sum_{j\in \n\setminus J}\left\|f_j(u_j)-f_j(v_j)\right\|_2^2\bigg)^{\frac12}\\\ge
\bigg(\sum_{j\in J} \|u_j-v_j\|_1^2\bigg)^{\frac12}+\bigg(\sum_{j\in \n\setminus J}
\left\|u_j-v_j\right\|_1^2\bigg)^{\frac12}\ge \|u-v\|_{\ell_2^n(\ell_1)},
\end{multline}
where in the first inequality of~\eqref{eq:phi in image} we used the leftmost inequality in~\eqref{eq:fj assumption}. The corresponding upper bound is deduced as follows from Cauchy--Schwarz,
the rightmost inequality in~\eqref{eq:fj assumption}, the definition of $J$ in~\eqref{eq:def J}, and the upper bound on $|J|$ in~\eqref{eq:J upper}.
\begin{align*}
&\|\phi(u)-\phi(v)\|_{(\ell_1^J(\ell_1)\oplus L_1)_1}\\&=\sum_{j\in J} \|u_j-v_j\|_1+\bigg(\sum_{j\in \n\setminus J}\left\|f_j(u_j)-f_j(v_j)\right\|_2^2\bigg)^{\frac12}\\
&\lesssim \sqrt{J|}\bigg(\sum_{j\in J} \|u_j-v_j\|_1^2\bigg)^{\frac12}+\left(\max_{j\in \n\setminus J} \sqrt{\log |X_j|}\log\log |X_j|\right)\bigg(\sum_{j\in \n\setminus J}
\left\|u_j-v_j\right\|_1^2\bigg)^{\frac12}\\
&\le \sqrt[4]{\log |S|}\sqrt{\log\log|S|}\bigg(\sum_{j\in J} \|u_j-v_j\|_1^2\bigg)^{\frac12}+\frac{\sqrt[4]{\log |S|}\log\left(\frac{\sqrt{\log |S|}}
{\log\log |S|}\right)}{\sqrt{\log\log |S|}}
\bigg(\sum_{j\in \n\setminus J}
\left\|u_j-v_j\right\|_1^2\bigg)^{\frac12}\\
&\lesssim \sqrt[4]{\log |S|}\sqrt{\log\log|S|}\cdot \|u-v\|_{\ell_2^n(\ell_1)}.\tag*{\qedhere}
\end{align*}
\end{proof}

We shall next justify the upper bound on $c_1(\ell_q^n\left(\F_2^n,\|\cdot\|_p\right))$ in~\eqref{eq:discrete distortion sharp}. Recall that we are assuming here that $q>p\ge 1$. Since for every $x,y\in M_n(\F_2)$ we have
$$
\|x-y\|_{\ell_q^n(\ell_p^n)}=\bigg(\sum_{j=1}^n\|x_j-y_j\|_p^q\bigg)^{\frac{1}{q}}=\bigg(\sum_{j=1}^n\|x_j-y_j\|_1^{\frac{q}{p}}\bigg)^{\frac{1}{q}},
$$
by H\"older's inequality
\begin{equation}\label{eq:1/p snowflake}
\frac{\|x-y\|_{\ell_1^n(\ell_1^n)}^{\frac{1}{p}}}{n^{\frac{1}{p}-\frac{1}{q}}}\le \|x-y\|_{\ell_q^n(\ell_p^n)}\le \|x-y\|_{\ell_1^n(\ell_1^n)}^{\frac{1}{p}}.
\end{equation}
By a classical theorem of Bretagnolle, Dacunha-Castelle and Krivine~\cite{BDK65} (see also~\cite[Theorem~5.11]{WW75}), for every $\alpha\in (0,1]$ the metric space $(L_1,\|x-y\|_1^\alpha)$ admits an isometric embedding into $L_1$. Hence, the metric space $$\left(M_n(\F_2),\|x-y\|_{\ell_1^n(\ell_1^n)}^{\frac{1}{p}}\right)$$ admits an isometric embedding into $L_1$, and consequently~\eqref{eq:1/p snowflake} implies that
\begin{equation*}
c_1\left(\ell_q^n\left(\F_2^n,\|\cdot\|_p\right)\right)\le n^{\frac{1}{p}-\frac{1}{q}}. \tag*{\qed}
\end{equation*}

\begin{remark}\label{rem:mn proof}
Arguing similarly to the above discussion also justifies the assertion in Remark~\ref{rem:mn intro}. Indeed, suppose that $m,n\in \N$ and $q>p\ge 1$. Then by H\"older's inequality for every $x,y\in \ell_q^m(\F_2^n)$,
\begin{equation}\label{eq:snowflake 1}
\frac{\|x-y\|_{\ell_1^m(\ell_1^n)}^{\frac{1}{p}}}{m^{\frac{1}{p}-\frac{1}{q}}}= \frac{\|x-y\|_{\ell_p^m(\ell_p^n)}}{m^{\frac{1}{p}-\frac{1}{q}}} \le \|x-y\|_{\ell_q^m(\ell_p^n)}\le \|x-y\|_{\ell_p^m(\ell_p^n)}=\|x-y\|_{\ell_1^m(\ell_1^n)}^{\frac{1}{p}}.
\end{equation}
Also, because for every $x,y\in \ell_q^m(\F_2^n)$ and $j\in \n$ we have $\|x_j-y_j\|_1\in \{0,\ldots,n\}$,
\begin{equation}\label{eq:snowflake 2}
\|x-y\|_{\ell_q^m(\ell_p^n)}=\bigg(\sum_{j=1}^n \|x_j-y_j\|_1^{\frac{q}{p}}\bigg)^{\frac{1}{q}}\in \left[\|x-y\|_{\ell_1^m(\ell_1^n)}^{\frac{1}{q}},n^{\frac{1}{p}-\frac{1}{q}}\|x-y\|_{\ell_1^m(\ell_1^n)}^{\frac{1}{q}}\right].
\end{equation}
Since the metric spaces $$\left(M_n(\F_2),\|x-y\|_{\ell_1^n(\ell_1^n)}^{\frac{1}{p}}\right)\qquad \mathrm{and} \qquad \left(M_n(\F_2),\|x-y\|_{\ell_1^n(\ell_1^n)}^{\frac{1}{q}}\right)$$ admit an isometric embedding into $L_1$, it follows from~\eqref{eq:snowflake 1} and~\eqref{eq:snowflake 2} that
\begin{equation}\label{eq:min}
c_1\left(\ell_q^m\left(\F_2^n,\|\cdot\|_p\right)\right)\le \min\left\{m^{\frac{1}{p}-\frac{1}{q}},n^{\frac{1}{p}-\frac{1}{q}}\right\}.
\end{equation}
Since $\ell_q^m\left(\F_2^n,\|\cdot\|_p\right)$ contains an isometric copy of $\ell_q^{\min\{m,n\}}\left(\F_2^{\min\{m,n\}},\|\cdot\|_p\right)$, by~\eqref{eq:discrete distortion sharp} and~\eqref{eq:min},
$$
c_1\left(\ell_q^m\left(\F_2^n,\|\cdot\|_p\right)\right)\asymp c_1\left(\ell_q^{\min\{m,n\}}\left(\F_2^{\min\{m,n\}},\|\cdot\|_p\right)\right),
$$
as required.
\end{remark}

We end with a brief justification of~\eqref{eq:cont dist}. If $1\le p<q$ and $p\le 2$ then $c_1\left(\ell_q^n(\ell_p^n)\right)\gtrsim n^{1/p-1/q}$, as proved by Kwapie{\'n} and Sch\"utt~\cite{KS89}. The reverse inequality follows from the fact that the $\ell_q^n(\ell_p^n)$ norm is $n^{1/p-1/q}$-equivalent to the $\ell_p^n(\ell_p^n)$ norm, and from the fact~\cite{Kad58}  that $\ell_p$ is isometric to a subspace of $L_1$ when $p\le 2$. When $q>p>2$,  the $\ell_p^n$ norm is $n^{1/2-1/q}$-equivalent to the $\ell_2^n$ norm and the $\ell_q^n$ norm is $n^{1/2-1/q}$-equivalent to the $\ell_2^n$ norm. So, the  $\ell_q^n(\ell_p^n)$ norm is $n^{1-1/p-1/q}$-equivalent to the $\ell_2^n(\ell_2^n)$ norm, which embeds isometrically into $L_1$. For the matching lower bound, suppose that $T:\ell_q^n(\ell_p^n)\to L_1$ is an injective linear mapping. Since $L_1$ has cotype $2$ (see e.g.~\cite{Mau03}),
\begin{multline}\label{eq:use rademacher cotype}
\frac{n^2}{\|T^{-1}\|^2}\le \sum_{j=1}^n\sum_{k=1}^n \|Te_{jk}\|_1^2\lesssim \frac{1}{2^{n^2}}\sum_{\e\in \{-1,1\}^{n^2}}\Big\|\sum_{j=1}^n\sum_{k=1}^n \e_{jk}Te_{jk}\Big\|_1^2\\
\le \frac{\|T\|^2}{2^{n^2}}\sum_{\e\in \{-1,1\}^{n^2}}\Big\|\sum_{j=1}^n\sum_{k=1}^n e_{jk}\Big\|_{\ell_q^n(\ell_p^n)}^2=\|T\|^2\cdot  n^{\frac{2}{p}+\frac{2}{q}}.
\end{multline}
By~\eqref{eq:use rademacher cotype} we have $\|T\|\cdot\|T^{-1}\|\gtrsim n^{1-1/p-1/q}$. The fact that $c_1\left(\ell_q^n(\ell_p^n)\right)\gtrsim n^{1-1/p-1/q}$ now follows by a standard differentiation argument; see e.g.~\cite[Chapter~7]{BL00} (alternatively, one could repeat the above argument mutatis mutandis, while using the fact that $L_1$ has metric cotype $2$ directly; see~\cite{MN08}).

\bibliographystyle{abbrv}
\bibliography{pythagorean}

 \end{document}